\documentclass[11pt]{article} 
\usepackage[utf8]{inputenc}
\usepackage[T1]{fontenc}
\usepackage[twoside,left=3.66cm,right=3.66cm,bottom=3cm,top=2.66cm]{geometry}
\usepackage{microtype}
\usepackage{amsmath,amsthm,amsfonts,amssymb}
\usepackage{mathrsfs}
\usepackage{cite}
\usepackage{enumerate}
\usepackage[all,2cell]{xy} \SilentMatrices
\usepackage{pdftricks}
\usepackage{epsfig,color}
\usepackage{mathtools}

\usepackage{framed}
\usepackage{color}
\definecolor{shadecolor}{gray}{0.875}
\definecolor{col}{RGB}{42, 95, 151}
\usepackage[colorlinks=true, linkcolor=col, citecolor=col, filecolor = col, menucolor = col, urlcolor = col]{hyperref}

\theoremstyle{plain}
\newtheorem{theorem}{Theorem}[section]
\newtheorem*{lemma*}{Lemma}
\newtheorem{lemma}[theorem]{Lemma}
\newtheorem*{theorem*}{Theorem}
\newtheorem{proposition}[theorem]{Proposition}
\newtheorem*{proposition*}{Proposition}
\newtheorem{corollary}[theorem]{Corollary}
\newtheorem*{corollary*}{Corollary}
\newtheorem{conj}{Conjecture}
\theoremstyle{definition}
\newtheorem{remark}[theorem]{Remark}
\newtheorem*{remark*}{Remark}
\newtheorem*{definition*}{Definition}

\newtheorem*{example*}{Example}
\newtheorem{example}[theorem]{Example} 

\newtheorem*{question*}{Question}

\def\min{\operatorname{min}}

\def\max{\operatorname{max}}

\def\c1{\operatorname{c_1}}
\def\c2{\operatorname{c_2}}

\def\CC{{\mathbb C}}
\def\ZZ{{\mathbb Z}}

\def\QQ{{\mathbb Q}}
\def\PP{{\mathbb P}}

\def\D{{\mathcal D}}

\def\O{{\mathcal O}}

\def\H{{\mathscr H}}

\def\M{{\mathcal M}}

\def\g{\mathfrak g}

\def\+{\oplus}                   
\def\*{\otimes}

\def\Pic{\operatorname{Pic}}
\def\Hom{\operatorname{Hom}}

\def\Supp{\operatorname{Supp}}

\def\Supp{\operatorname{Supp}}

\def\Def{\operatorname{Def}}

\renewcommand\Delta{B}

\usepackage{sectsty}

\title{Curve classes on irreducible holomorphic\\ symplectic varieties}
\author{Giovanni Mongardi and John Christian Ottem}
\date{}

\begin{document}

\maketitle
\thispagestyle{empty}
\vspace{-0.7cm}
\begin{abstract}
We prove that the integral Hodge conjecture holds for 1-cycles on irreducible holomorphic symplectic varieties of K3 type and of Generalized Kummer type. As an application, we give a new proof of the integral Hodge conjecture for cubic fourfolds.
\end{abstract}

\vspace{0.4cm}
\thispagestyle{empty}


\def\X{\mathcal X}
\def\alg{\text {alg}}
Let $X$ be a smooth complex projective variety of dimension $n$. Write $H^{p,p}(X,\ZZ)=H^{2p}(X,\ZZ)\cap H^{p,p}(X,\CC)$ for the group of integral degree $2p$ Hodge classes. We say that the {\em integral Hodge Conjecture} holds for $k$-cycles on $X$ if  $H^{n-k,n-k}(X,\ZZ)$ is generated by classes of $k$-dimensional algebraic subvarieties on $X$. While the usual Hodge conjecture predicts that this statement should be true with $\QQ$-coefficients, it is known that the integral Hodge conjecture can fail in general, even for 1-cycles \cite{benoistottem,kollar,totaro,VoisinHC}. For 1-cycles, the validity of the conjecture depends very much on the birational properties of $X$ (in particular the group of degree $2n-2$ Hodge classes modulo algebraic classes is a birational invariant). For instance, Voisin showed that the conjecture holds for threefolds that are either uniruled; or satisfy $K_X=0$ and $H^{2}(X,\O_X)=0$ \cite{Voisin}. For varieties with $K_X=0$ of higher dimension (in particular, abelian fourfolds), not much is known.

In this paper we prove that the integral Hodge conjecture holds for 1-cycles on certain irreducible holomorphic symplectic varieties. We consider varieties of K3-type (deformation equivalent to Hilbert schemes of points on K3 surfaces) and generalized Kummer type (see Section \ref{prelim}). Our main theorem is the following:

\begin{theorem}\label{IHC}
Let $X$ be a projective holomorphic symplectic variety of K3--type or of generalized Kummer type. Then the integral Hodge conjecture holds for 1-cycles on $X$.
\end{theorem}In fact, the group $H^{2n-1,2n-1}(X,\ZZ)$ is generated by classes of rational curves. This can be extended as follows:
\begin{theorem}\label{integralmori}
Let $X$ be a projective holomorphic symplectic variety of K3 type or of generalized Kummer type. Then the semigroup of effective curve classes is generated (over $\mathbb{Z}$) by classes of rational curves.
\end{theorem}
The above theorems apply in particular to the variety of lines on a cubic fourfold. In this setting, Shen \cite{Shen} proved that the integral Hodge conjecture is related to the algebraicity of the Beauville--Bogomolov form (see Section \ref{applications}). Using the incidence correspondence, we also give a new proof of the following result of Voisin:
\begin{corollary}\label{cubic4fold}
The integral Hodge conjecture holds for 2-cycles on cubic fourfolds.
\end{corollary}
The proofs in this paper rely on several results and constructions that were already in the literature. In particular, Theorems \ref{IHC} and \ref{integralmori} involve a deformation argument similar to that in \cite{AV} and \cite{CP}. We first consider the Hilbert scheme of a K3 or a generalized Kummer variety, where we exhibit special families of rational curves that represent primitive classes in $H^{2n-2}(X,\ZZ)$, and which also deform in their Hodge loci. This then in turn implies that any integral degree $2n-2$ Hodge class on a deformation is represented by a rational curve.
%
%

We would like to thank O. Benoist, G. Pacienza, M. Shen and C. Vial for useful discussions. GM was supported by the project ``2013/10/E/ST1/00688'' and ``National Group for Algebraic and Geometric Structures, and their Application'' (GNSAGA - INdAM), and JCO was supported by an RCN grant  no. 250104.

\section{Preliminaries}\label{prelim}
We work over the complex numbers. An {\em irreducible holomorphic symplectic variety} (IHS) $X$ is a simply connected Calabi--Yau manifold carrying a non-degenerate holomorphic 2-form $\omega$ generating $H^{2,0}(X)$. $X$ carries the {\em Beauville--Bogomolov form} $q$, which is a non-degenerate quadratic form on $H^2(X,\ZZ)$. 

\subsection{Varieties of K3--type}
Let $S$ be a complex K3 surface and let $S^{[n]}$ denote its Hilbert scheme (or Douady space in the non-projective setting) of length $n$ subschemes of $S$. This is a smooth IHS variety of dimension $2n$. In general, we say that an IHS variety is of {\em K3-type} if it is deformation equivalent to $S^{[n]}$. For these varieties, the Beauville--Bogomolov form $q$ has signature $(3,20)$.

The (co)homology groups of $X=S^{[n]}$ of degree 2 and $4n-2$ are easy to describe. We have the Hilbert--Chow morphism $HC:S^{[n]}\to S^{(n)}$ which induces an injection $H^{2}(S^{(n)},\ZZ)\to H^2(X,\ZZ)$ and an injective map $i:H^2(S,\ZZ)\to H^2(X,\ZZ)$ obtained by symmetrizing a line bundle on $S$. From these we obtain decompositions
\begin{equation}
H^2(X,\ZZ)=H^2(S,\ZZ)\oplus \ZZ \Delta
\end{equation}\label{homologySm}and
\begin{equation}\label{decomphom}
H_2(X,\ZZ)=H_2(S,\ZZ)\oplus \ZZ \tau
\end{equation}where $\Delta=\frac12[E]$ is one half of the class of the exceptional divisor $E$ of $HC$ and $\tau$ is the class of a rational curve in a fiber of $HC|_E$. In particular, we see that the integral Hodge conjecture holds automatically for 1-cycles on $X$, since it holds on $S$.



Viewing $H_2(X,\ZZ)$ as $\Hom(H^2(X,\ZZ),\ZZ)$, we see that $q$ defines an embedding of lattices
\begin{equation}\label{duality}
\phi:H^2(X,\ZZ)\to H_2(X,\ZZ)=H^{4n-2}(X,\ZZ).
\end{equation}Over $\QQ$ this defines an isomorphism $\phi:H^2(X,\QQ)\to H_2(X,\QQ)$, and we extend $q$ to a $\QQ$-valued quadratic form on $H_2(X,\QQ)$ and $H_2(X,\ZZ)$. In this setting both of the decompositions above are orthogonal with respect to the form $q$.  For a divisor $D$, Fujiki's relation states that $D^{2n}=cq(D)^{n}$ for a rational constant $c>0$. This implies that the two integral degree $4n-2$ classes $D^{2n-1}$ and $\phi([D])$ are proportional in $H^{4n-2}(X,\ZZ)$ (unless $D^{2n}=0$). In particular, this means that the Hodge locus of $D$ coincides with the Hodge locus of any curve class proportional to it under the map $\phi$.

%
%
%
%

\subsection{Generalized Kummer varieties}
Let $S$ denote an abelian surface and let $n$ be a positive integer. We define the {\em generalized Kummer variety} $X=K_n(S)\subset S^{[n+1]}$ associated to $S$ as the fiber over  0 of the summation map $S^{[n+1]}\to S$. Smooth deformations of these varieties are said to be of generalized Kummer type. 

As in the previous case, we have a decomposition $$H^2(X,\ZZ)=H^2(S,\ZZ)\oplus \ZZ e$$ where $e=\frac 12 E$ where $E$ is the restriction of the Hilbert--Chow divisor on $S^{[n+1]}$, and 
$$H_2(X,\ZZ)=H_2(S,\ZZ)\oplus \ZZ \eta$$ where $\eta$ is the class of a minimal curve in the fibres of $HC_{|E}$. Both of these are  orthogonal with respect to the Beauville--Bogomolov form. 

\subsection{Deforming rational curves}

Let $X$ be an IHS variety of dimension $2n$ and let $f\,:\,\mathbb{P}^1\rightarrow X$ be a non-constant map. Let $R$ denote the image of $f$ and let $\Def(X,[R])$ be the sublocus of the local deformation space $\Def(X)$ where the class $[R]$ remains Hodge. The deformation theory of $R\subset X$ and the map $f$ is well-understood, by results of Ran \cite{Ran} and later, in this particular case, by Amerik and Verbitsky \cite{AV} and by Charles and Pacienza \cite{CP}. We can formulate the result we need using the Kontsevich moduli space $\overline \M_0(X,\beta)$ parameterizing stable maps $f:\PP^1\to X$ with image of class $f_*[\PP^1]=\beta$. If $f:\PP^1\to X$ is a finite map, then every component $\M\subseteq \overline {\M}_0(X,[R])$ containing the corresponding point $[f]$ has dimension  at least $2n-2$. Ran's results can then be summarized as follows:
\begin{proposition}{\cite[Corollaries 3.2, 3.3 and 5.1]{Ran}}\label{ranprop}
Let $X,R,f$ be as above. Suppose there is a component of $\overline{\mathcal{M}}_0(\mathbb{P}^1,[R])$ of dimension $2n-2$ containing $[f]$. Then the curve $R$ deforms in the Hodge locus $\Def(X,[R])$.
\end{proposition}In other words, given a family $\pi:\X\to T$ of IHS varieties with a special fiber $X=\X_0$; a rational curve $f:\PP^1\to X$ with image $R\subset X$; and a global section of $R^{4n-2}\pi_*\ZZ$ of Hodge type $(2n-1,2n-1)$, specializing to $[R]$ on $X$. Then if there is a component of the moduli space $\overline{\M}_0(X,[R])$ containing $[f]$ of dimension exactly $2n-2$, then the map $f:\PP^1\to \X$ deforms (after taking some finite cover of $T$), and in particular $R$ deforms in the fibers of $\pi$. (See also \cite[Section 3]{CP} and \cite{AV} for similar statements).



\section{Special rational curves}

To prove Theorems \ref{IHC} and \ref{integralmori}, we will need to construct certain special rational curves on $S^{[n]}$ and $K_n(S)$ that satisfy the conditions of Proposition \ref{ranprop}. To explain the basic idea, consider the case where $n=2$, and $S$ is a K3 surface of degree 2. Letting $H$ be the polarization, each smooth $C\in |H|$ is a genus 2 curve admitting a unique $\g^1_2$. This  defines a rational curve $R_C$ on $C^{[2]}$, and hence on $S^{[2]}$. We can write the class $[R_C]\in H_2(S^{[2]},\ZZ)$ in terms of the decomposition \eqref{decomphom}; intersecting $R_C$ with $H$ shows that it has the form $H-t B$ for some $t\in \ZZ$. In particular, the class is primitive. From the double cover $S\to \PP^2$, we also obtain a plane $\PP^2\subset S^{[2]}$, which contains all the curves $R_C$, and the plane can be contracted by a birational map. In particular $R_C$ deforms in a family of dimension $2$ and thus satisfies Proposition \ref{ranprop}.

The example above is special in the sense that the curves $R_C$ are smooth. To construct other classes on $S^{[n]}$, we need to consider singular curves on $S$ and linear series on their normalizations. For this, we apply results on Brill--Noether theory on nodal curves on surfaces due to Ciliberto and Knutsen \cite{CK} for $K3$ surfaces, and by Knutsen, Lelli-Chiesa and the first author \cite{KLM} for abelian surfaces. 

Let us start with Hilbert schemes of K3 surfaces. Let $(S,H)$ be a primitively polarized K3 surface of degree $H^2=2p-2$ and let $C\in |H|$ be a curve with $\delta$ nodes as its only singular points. Given a linear series $\g^1_n$ on the normalization $\widetilde C$ of $C$, we obtain a natural rational curve $R_C$ in $S^{[n]}$ via the incidence correspondence
\begin{equation}\label{incidence}
I=\{(P,[Z])\in S\times S^{[n]} | P\in \Supp(Z)\} \to S^{[n]}.
\end{equation}
We say that the nodes are `non-neutral' with respect to the linear series if the linear series is base point free, and the corresponding morphism $\widetilde C\to \PP^1$ has simple ramification and does not ramify over the nodes of $C$. If this genericity condition is satisfied, we have that the homology class of $R_C$ is given in terms of the decomposition \eqref{decomphom} as

\begin{equation}\label{curveindecomp}
H-(p-\delta+n-1)\tau,
\end{equation}
 where $H\in \Pic(S)$ (see \cite{KLM}). Note that this class is primitive, since $H$ is. Moreover, the following divisor is proportial to $R_C$ via \eqref{duality}: $$(2n-2)H-(p-\delta+n-1)B.$$
 


The main properties of these curves are summarized in the following theorem. In the theorem, $V^n_{|H|,\delta}\subseteq |H|$ denotes the Severi variety of curves $C$ with $\delta$ nodes, whose normalizations admit a $\g_n^1$.

\begin{theorem}\cite[Thm 0.1]{CK} \label{thm:exist_K3}
 Let $(S,H)$ be a very general primitively polarized $K3$ of genus $p:=p_a(H)\ge 2$. Let $\delta$ and $n$ be integers satisfying $0 \leq \delta \leq p$ and $n \geq 2$.  Then the following statements hold:
  \begin{itemize}
\item [(i)] $V^n_{|H|,\delta}$ is non-empty if and only if
\begin{equation} \label{eq:boundA}
\delta \geq \alpha\Big(p-\delta-(n-1)(\alpha+1)\Big), \text{ where } \alpha= \Big\lfloor \frac{p-\delta}{2n-2}\Big\rfloor.
\end{equation}
\item [(ii)]  Whenever non-empty, $V^n_{|H|,\delta}$ is equidimensional of the expected dimension $\min\{2n-2,p-\delta\}$, and a general point on each component corresponds to an irreducible curve with normalization $\widetilde{C}$ of genus $g=p-\delta$, such that the set of $\g^1_n$'s on $\widetilde C$ is of dimension $\max\{0,2n-2-g\}$.
\item [(iii)] There is a component $V\subseteq V^n_{|H|,\delta}$, so that for $C$ and $\widetilde C$ as in $(ii)$, a general $\g^1_{n}$ on $\widetilde{C}$ is base point free and all nodes of $C$ are non-neutral with respect to it.
\end{itemize}
\end{theorem}

So if $p-\delta\ge 2n-2$, $V^n_{|H|,\delta}$ has dimension $2n-2$ and the set of $\g^1_n$ on a general curve is 0-dimensional. When $p-\delta\le 2n-2$, $V^n_{|H|,\delta}$ has dimension $p-\delta$ and the set of rational curves $R_C$ on $S^{[2]}$ has dimension $(p-\delta)+(2n-2-g)=2n-2$. Hence in either case we have, if the above inequality is satisfied, a family of rational curves on $S^{[n]}$ of dimension $2n-2$. 


There is a very similar statement in the generalized Kummer case. We consider a polarized abelian surface $(S,H)$ of degree $H^2=2p-2$. Let $C$ be an element of $\{H\}$, the continuous system of curves with cohomology class $H$ (that is, $|H|$ up to translations on $S$). As before, given a linear series $\g^1_{n+1}$ on the normalization of $C$, we obtain a natural rational curve $R_C$ in $K_n(S)$. If the nodes are non-neutral with respect to this series, the class of $R_C$ is given by

$$H-(p-\delta+n)\eta.$$This is again primitive, if $H$ is.


\begin{theorem}\cite[Theorem 1.6]{KLM} \label{thm:exist_GK}
 Let $(S,H)$ be a very general abelian surface of genus $p:=p_a(H)\ge 2$. Let $\delta$ and $n$ be integers satisfying $0 \leq \delta \leq p-2$ and $n \geq 2$.  Then the following hold:
  \begin{itemize}
\item [(i)] There exists a nodal curve in $\{H\}$ whose normalization has a linear series of type $\g^1_{n+1}$ if and only if 
\begin{equation} \label{eq:boundA_abel}
\delta \geq \alpha\Big(p-\delta-1-(n+1)(\alpha+1)\Big), \text{ where }
\alpha= \Big\lfloor \frac{p-\delta-1}{2n+2}\Big\rfloor; 
\end{equation}
\item [(ii)]  When non-empty, the set of $\g^1_{n+1}$'s on curves in $\{H\}$ with $\delta$ nodes is equidimensional of dimension  $\min\{p-\delta,2n\}$ and a  general element in each component is an irreducible curve $C$ with normalization $\widetilde{C}$ of genus $g:=p-\delta$ such that it has a $\max\{0,\rho(g,1,n+1)\}=2n-g$ dimensional set of $\g^1_{n+1}$;
\item [(iii)]  There is at least one component of the above locus where, for $C$ and $\widetilde{C}$ as in (ii), a general $\g^1_{n+1}$ on $\widetilde{C}$ is base point free and all nodes of $C$ are non-neutral with respect to it.
\end{itemize}
\end{theorem}
Let us illustrate how to use this theorem to construct rational curves in a few interesting cases on Hilbert schemes of $n$ points on K3 surfaces. 

\begin{example}First, let us work with curves $R$ such that their dual divisor $D$ has square $8d-(2n-2)$ and divisibility $2$ (that is, the pairing between $D$ and the rest of $H^2(X,\ZZ)$ takes all even integer values)\footnote{An easy computation shows that the square of any divisor of divisibility $2$ is congruent to $(2-2n)$ modulo 8, so we are actually considering all possible squares for $D$}. It is possible to deform to the case where $X=S^{[n]}$ and $D=2H-B$ (see Lemma \ref{allthemuweneed} below), where $H$ is the divisor induced by the polarization on $S$ and $H^2=2d$. The dual curve has class $H-(n-1)\tau$ is represented by a rational curve  corresponding to a linear series of type $\g^1_n$ on the normalization of a rational curve in $|H|$. It is well-known that such rational curves always exists if $d>0$ (and thus if the square of $2H-B$ is positive) and so we obtain the desired linear series exists on its normalization.
\end{example}

\begin{example}
A similar situation arises with the curve class $H-(2n-1)\tau$ on $S^{[n]}$. This corresponds to linear series of type $\g^1_n$ on normalizations of curves on $S$ of geometric genus $n$ inside $|H|$. If $q(H-(2n-1)\tau)>0$, the square of $H$ is strictly bigger than $2n-2$, so there does indeed exist curves of genus $n$ inside $|H|$ (here the integer $\alpha$ above is zero, and the conditions of Theorem \ref{thm:exist_K3} are satisfied). Since any smooth curve of genus $n$ contains a $\g_n^1$, we see obtain the desired rational curves on $S^{[n]}$.
\end{example}

\begin{example}
Let $F$ be the Fano variety of lines of a cubic fourfold $Y\subset \PP^5$. It is well-known that $F$ is deformation equivalent to a Hilbert scheme $S^{[2]}$ of a degree 14 K3 surface $(S,H)$. In terms of $H$ and $B$, the Pl\"ucker polarization $\O_{F}(1)$ from the ambient Grassmannian $Gr(2,6)$ is given by $\O(1)=2H-5B$ (see \cite{BD}). The dual curve class is given by $H-5\tau$. We can produce rational curves on $S^{[2]}$ with this class using Theorem \ref{thm:exist_K3} by taking $p=8, \delta=4$. Then $\alpha=2$ and the inequality \eqref{eq:boundA} is satisfied. In this example, we obtain a 2-dimensional family of rational curves, since $V^n_{|H|,\delta}$ has dimension $2$, and each normalization $\widetilde C$ has at most one $\g^1_2$.

These rational curves have degree $(2H-5B)\cdot (H-5\tau)=2\cdot 14-25=3$. Geometrically the curves arise as follows: For each line $\ell$ of `type II' (i.e., with normal bundle $N_\ell=\O(1)^2+\O(-1)$), there is a tangent $\PP^3$ which intersects $Y$ in a cubic surface singular along $\ell$; the residual lines are parameterized by a degree 3 curve in $F$.
\end{example}


More generally, we have the following result for K3 surfaces, with a straightforward generalization to the generalized Kummer case (with $2n+3\leq \mu\leq 3n+3$):


\begin{lemma}\label{allthecurvesweneed}
Let $S$ be a K3 surface of degree $2d$ and let $2n-1\leq \mu\leq 3n-3$ be an integer. Let $R$ be the curve class $H-\mu\tau$ in $S^{[n]}$. Suppose $q(R)\geq 0$. Then there exists a rational curve of class $R$ and a component $\M\subseteq\overline{\mathcal M}_0(X,[R])$ of dimension $2n-2$.
\end{lemma}
\begin{proof}
We want to construct $R$ as in the previous examples. That is, we are trying to construct linear series of type $\g^1_n$ on curves in $|H|$ of geometric genus $\mu-n+1$. By the bound on $\mu$, we have $\alpha=0$ unless $n=2$ or $\mu=3n-3$. Therefore, if we are not in these cases, the bound of Theorem \ref{thm:exist_K3} is satisfied and the linear series exists if and only if there are curves in $|H|$ of genus $\mu-(n-1)$. Since $H$ has genus $d+1$, this happens if and only if $d+1\geq \mu-(n-1)$. However, by assumption, $q(R)=2d-\frac{\mu^2}{2n-2}$ is positive, and this implies $d+1\geq \mu-(n-1)$. The remaining cases when $n=2$ or $\mu=3n-3$ have $\alpha=1$, but the condition of Theorem \ref{thm:exist_K3} is trivially satisfied, so the above applies also to this case.
By construction, the curve $R$ deforms in a family of dimension $2n-2$. Moreover, the incidence correspondence \eqref{incidence} can be used to prove that all deformations of these rational curves on the Hilbert scheme of points on a general K3 (or a generalized Kummer) are actually induced by linear series on different curves on $S$ (see \cite[Proposition 5.3]{KLM} or \cite[Proposition 3.6]{KLM2} for a proof of this). \end{proof}

%
%

In particular, this implies that all curves constructed using the above theorems satisfy the hypothesis of Proposition \ref{ranprop}. This fact will be important in the proof of Theorem \ref{IHC} and Lemma \ref{sumofrat} below. Moreover, we will construct similar curves to prove Lemma \ref{ratzero} and then finally Theorem \ref{integralmori}. 

\begin{remark}
These proofs of the two theorems above rely on the assumption that the Picard number of $S$ is 1. If $S$ has higher Picard number, these curves can move in larger families, as was exploited in \cite{FKP} to construct rational surfaces.
\end{remark}
Finally, let us conclude with a lattice theoretic result which highlights the importance of Lemma \ref{allthecurvesweneed}. Notice that an analogous result was proven in \cite[Theorem 2.4]{CP}, with different values of $\mu$. We include a proof for completeness.
Let $L$ be an even lattice and let $A_L:=L^\vee/L$ be its discriminant group, which is a finite group. For the properties of the discriminant group and its links to lattice theory, we refer to \cite{Nik}. Given a primitive element $l\in L$, we define the divisibility $div(l)$ as the positive generator of the ideal $(l,L)$ in $\ZZ$. The element $l/div(l)$ is then a well defined element of $L^\vee$ and we denote by $[l/div(l)]$ its class in $A_L$. We will use the following, known as Eichler's criterion:
\begin{lemma}\label{eichler}\cite[Lemma 3.5]{ghs}
Let $L'$ be an even lattice and let $L=U^2\oplus L'$. Let $v,w\in L$ be two primitive elements such that the following holds:
\begin{itemize}\renewcommand{\labelitemi}{$\bullet$}
\item $v^2=w^2$.
\item $[v/div(v)]=[w/div(w)]$ in $A_L$.
\end{itemize}
Then there exists an isometry $g\in \widetilde{O}^+(L)$ with determinant one and such that $g(v)=w$.
\end{lemma}
Here, $U$ denotes the rank 2 hyperbolic lattice, $O^+(L)$ denotes the group of isometries preserving the orientation on $L$ and $\widetilde{O}^+(L)$ denotes the subgroup of isometries whose induced action on $A_L$ is trivial. Note that, if $L$ is primitively embedded into another lattice $M$, an isometry of $\widetilde{O}(L)$ can be extended to an isometry of $M$ acting trivially on $L^\perp$.
\begin{lemma}\label{allthemuweneed}
Let $(X,D)$ be a pair consisting of a manifold $X$ of $K3^{[n]}$-type and a divisor $D$ of square $2d$ and divisibility $t$. Then there exists a polarized $K3$ surface $(S,H)$ of degree $2s$, an integer $2n-1\leq \mu\leq 3n-3$ such that $(X,D)$ is deformation equivalent to $(S^{[n]},tH-\mu/e B)$, where $e=\gcd(2n-2,\mu)$ and the dual curve to $tH-\mu/e B$ is $H-\mu\tau$.
\end{lemma}
\begin{proof}
First of all, the pair $(X,D)$ can be deformed to a pair $(S^{[n]},D')$ for some $S$ and some $D'$, as the locus of Hilbert schemes is an hyperplane divisor in the marked moduli space of manifolds of $K3^{[n]}$ type, hence it intersects the image of the moduli space of pairs $(X,D)$ in the period domain, which is path connected, and by \cite[Proposition 7.1]{Mark_tor} this path lifts to a deformation of $(X,D)$ to $(S^{[n]},D')$. We can further deform $S$ so that it is very general and has Picard rank one, 
with generator $H$, therefore $D'=aH+bB$, where $t$ divides $a$ and $D'^2=D^2$.  
By \cite[Corollary 9.5]{Mark_tor}, there is an embedding of $\Lambda_n:=H^2(S^{[n]},\ZZ)$ in the Mukai lattice $\Lambda:=U^4\oplus E_8(-1)^2$,  which is well defined up to isometry and coincides with the usual embedding obtained by seeing $S^{[n]}$ as a moduli space of ideal sheaves on $S$. Notice that, if $n-1$ is a prime power, there is only one isometry class of these embeddings. Moreover, the orthogonal complement of $\Lambda_n$ under this embedding has rank one, and let $v$ be a generator of it. 


Let us denote by $T(D')$ and $T(tH-\mu/eB)$ the rank two primitive lattices containing $D'$ and $v$ or $tH-\mu/eB$ and $v$ respectively. By \cite[Proposition 1.6]{apostolov}, the two pairs are deformation equivalent if and only if there exists an isometry of $T(D')$ and $T(tH-\mu/eB)$ sending $D'$ to $tH-\mu/eB$. That is, the isometry must send $v$ into $\pm v$ and therefore descends to an isometry of $\Lambda_n$. Let us consider the discriminant group $A_{\Lambda_n}:=\Lambda_n^\vee/\Lambda_n$, which is a cyclic group of order $2n-2$ generated by $[B/(2n-2)]$. By elementary lattice theory, the isometry above has to act as $\pm1$ on the discriminant group $A_{\Lambda_n}$, see e.g. \cite[Corollary 9.5]{Mark_tor}. Notice that $D'/t$ and $(tH-\mu/eB)/t=(H-\mu/(2n-2)B)$ are well defined elements of $\Lambda_n^\vee$ by the definition of divisibility.   

Notice that $[(tH-\mu/eB)/t]$ is equal to $\mu[B/(2n-2)]$ inside $A_{\Lambda_n}$. From Eichler's criterion in Lemma \ref{eichler}, if $[D'/t]=[(tH-\mu/eB)/t]$ there exists an isometry of $\Lambda_n$ which acts trivially on $A_{\Lambda_n}$ (and therefore can be extended to $\Lambda$) sending $D'$ to $tH-\mu/eB$. Therefore our claim holds true with $2n-1\leq \mu\leq 4n-4$, so that $\mu$ takes all possible values modulo $2n-2$. To obtain the desired bound, it suffices to compose with the isometry given by reflection along $B$, which acts as $-1$ on $A_{\Lambda_n}$, so that only half of the values of $\mu$ are needed.
\end{proof}

%
%
%
%

\section{Proof of Theorem \ref{IHC}}

Let $X,X'$ be two IHS varieties and let $h,h'$ be primitive polarizations on $X$ and $X'$ respectively. We will for simplicity say that $(X,h)$ and $(X',h')$ are deformation equivalent if there is a smooth projective family $\pi: \X\to T$ over an irreducible curve $T$; a line bundle $\H$ on $X$; and two points $0,1\in T$ so that $(X,h)=(\X_0,c_1(\H)|_{\X_0})$ and $(X',h')=(\X_1,c_1(\H)|_{\X_1})$.


\begin{lemma}\label{ratcurve}
Let $(X,h)$ and $(X',h')$ be two deformation equivalent primitively polarized IHS varieties of dimension $2n$, connected by a family $\pi:\X\to T$, and let $R\subset X$ be a rational curve which deforms in a family of dimension $2n-2$ in $X$. Suppose further that $[R]$ is proportional to $h$ (via the embedding \eqref{duality}). Then $R$ deforms in the fibres of $\pi$. In particular, also $X'$ has an effective 1-cycle  (with class proportional to $h'$) with components being rational curves.
\end{lemma}

\begin{proof}
Let $f:\PP^1\to X$ denote the stable map corresponding to $R$, so that $[R]=f_*[\PP^1]$.  Since the deformation space has an irreducible component of dimension $2n-2=\dim X-2$, we have, by Proposition \ref{ranprop}, that the curve deforms in its Hodge locus. By assumption, our class $[R]$ is Hodge on all the fibers of $\pi$, since it is proportional to the restriction of the (1,1)-class $c_1(\H)$ on $\X$. Since $T$ is irreducible, this means that $R$ deforms to a chain of rational curves on $X'$.
\end{proof}
%
%

\begin{theorem}\label{sumofrat}
Let $X$ be a IHS variety of K3-type or generalized Kummer type of dimension $2n$, and let $\gamma\in H^{4n-2}(X,\ZZ)$ be a primitive integral Hodge class with $q(\gamma)>0$. Then $\gamma$ is cohomologous to a sum of rational curves. 
\end{theorem}

\begin{proof}
Let $\gamma$ be such a class and let  $D$ the primitive divisor proportional to it (under the embedding of \eqref{duality}). Since $q(D)>0$, there is a deformation $\pi:\X\to T$ of $X$ over an irreducible curve $T$; points $0,1\in T$; and a divisor class $\D$ on $\X$, so that $(\X_0,\D_0)=(X,D)$ and $(\X_1,\D_1)$ is either the Hilbert scheme of a K3 surface $S$ or a generalized Kummer variety of an abelian surface $S$ and we can suppose such a surface has Picard rank one. 


The divisor $\D_1$ can be chosen to be of the form $tH-\mu\tau$ for $H\in \Pic(S)$ primitive and integers $t,\mu$ satisfying $2n-1\leq \mu \leq 3n-3$ in the $K3^{[n]}$ case and $2n+3\leq \mu \leq 3n+3$ in the Kummer case by Lemma \ref{allthemuweneed}. 


 
 Therefore, we can use the curves constructed in Lemma \ref{allthecurvesweneed} so that all possible monodromy orbits are covered. Let us call this rational curve $R'$ on $\X_1$ so that $[R']$ is primitive and proportional to $\D_1$ (via the embedding \eqref{duality}). By construction, there is a corresponding component of $\overline{\mathcal M}_0(X,[R'])$ has dimension exactly $2n-2$. Since $\D_1$ is the restriction of $[\D]$ on $\X$, the class $[R]$ is Hodge in the fibers of $\pi$, and so Lemma \ref{ratcurve} shows that the map $f:\PP^1\to \X$ deforms in a family dominating $T$. In particular, this means that $R$ deforms to a 1-cycle $R'$ on $X$, all of whose components are rational. By construction, the class of $R$ this is a multiple of $\gamma$, and so by primitivity $[R]=\gamma$. This completes the proof of the theorem.
\end{proof}

With this, we can prove our main theorem:
\begin{proof}[Proof of Theorem 0.1]
The group of integral degree $4n-2$ Hodge classes is generated by primitive classes with positive Beauville--Bogomolov square, and these have algebraic representatives by the previous theorem.
\end{proof}

The proof of the above theorem can likely be applied to other situations. Indeed, the key ingredient of the proof is a statement similar to a conjecture of Voisin \cite[Conj. 3.1 and Remark 3.2]{voi_conj}, namely the following:

\begin{conj}\label{vois_strong}
Let $X$ be a projective IHS variety of dimension $2n$. Then there is a primitive rational curve on $X$ which moves in a $(2n-2)$-dimensional family.
\end{conj}

This conjecture has been proved for varieties of K3--type and Kummer type in \cite[Theorem 5.1]{MP2} following ideas contained in \cite{CP} and \cite{MP}. 

\begin{theorem}
Let $X$ be a IHS variety such that Conjecture \ref{vois_strong} holds for generic projective deformations of $X$. Then the integral Hodge conjecture holds for 1-cycles on $X$.
\end{theorem}
Indeed, starting with $X$ and an integral class $\gamma\in H^{2n-1,2n-1}(X,\ZZ)$ so that $\gamma$ is proportional to a primitive ample divisor class $H\in \Pic(X)$, we can take a very general projective deformation $(X',H')$ of $(X,H)$ so that $\Pic(X')=\ZZ H'$. By the above conjecture, there exists a primitive curve on $X'$ proportional to $H'$ (via the embedding \eqref{duality}). Moreover, this curve moves in a family of dimension exactly $2n-2$. Therefore, by Proposition \ref{ranprop} such a curve deforms to a 1-cycle on $(X,H)$ representing $\gamma$.

\begin{remark}
If $C$ is a rational curve which is the ruling of a uniruled divisor $D\subset X$, then $C$ moves in a family of exactly dimension $2n-2$. However, some of the primitive curves constructed above cannot cover divisors, as was shown in \cite[Appendix A.3]{oberdieck}. In Loc.Cit, there are necessary condition to ensure that a primitive rational curve covers a divisor, and some examples where this conditions are not met are provided. Nevertheless, the curves we construct in Lemma \ref{allthecurvesweneed} are sufficient to prove our claim.
\end{remark}


\begin{remark}
For a non-projective IHS manifold, it is of course not true that the integral degree $4n-2$ Hodge classes are generated by classes of curves. This fails already in dimension two; there exists K\"ahler K3 surfaces with $H^{1,1}(X,\ZZ)=\ZZ \sigma$ for a class with self-intersection $\sigma^2=-4$ (which couldn't possibly be algebraic).
\end{remark}


\section{Proof of Theorem \ref{integralmori}}
To prove Theorem \ref{integralmori}, we need to analyze curves of non-positive Beauville--Bogomolov square. Let us consider first those of negative square. 
\begin{lemma}\label{ratnegative}
Let $Y$ be a smooth non-projective IHS manifold of dimension $2n$, with $\Pic(X)=\ZZ D$ and $q(D)<0$. Let $\gamma$ be the class of a curve on $Y$. Then $\gamma$ is a multiple of a rational curve which moves in a $2n-2$ dimensional family.
\end{lemma}
\begin{proof}
As $\gamma$ is effective and of negative square in a Picard rank one manifold, the divisor class $D$ is a wall divisor (or MBM class) in the sense of \cite[Definition 1.2]{mo_kahl}. Being a wall divisor is preserved by deformations in the Hodge locus of $\gamma$ by \cite[Theorem 1.3]{mo_kahl}, so we can take a projective small deformation $Y'$ of $Y$ such that $\gamma$ is contracted by a map $Y'\rightarrow X'$ by \cite[Theorem 2.5]{KLM2}. Now, this implies that there exists a map $\phi:Y\rightarrow X$ to a singular symplectic manifold $X$ which contracts $\gamma$ without taking any deformation (see \cite[Theorem 1.1]{BL}).

Let $F$ be a general fibre of the exceptional locus of $\phi$. By \cite[Theorem 1.3 (iii)]{Wie}, the normalization of any component of $F$ is a projective space $\PP^e$. Let $K$ be such a component and let $\eta: \PP^e\to K$ be the normalization map. Let $l\in H_2(\PP^e,\ZZ)$ be the class of a line. By hypothesis, we have $\gamma=a\eta_*(l)$ as a class in $H_2(Y,\QQ)$ for some $a\in \mathbb{Q}$. Here $a$ is actually an integer: $\eta^*(\gamma)\cdot l=a \eta^*\eta_*(l) \cdot l=a$. Hence $\gamma$ is an integral multiple of a rational curve.


Finally, if $f:\PP^1\to X$ is the map corresponding to $l$, we have by \cite[Lemma 9.4]{CMSB} that $N_f=\O_{\PP^1}(-2)\oplus \O_{\PP^1}(-1)^{e-1}\oplus \O_{\PP^1}^{2n-2e}\oplus \O_{\PP^1}(1)^{e-1}$. In particular, $f$ deforms in a family of dimension $h^0(N_f)-h^1(N_f)+1=2n-2$, as desired. (See also \cite[Proposition 3]{BHT}).
\end{proof}

\begin{lemma}\label{ratzero}
Let $Y$ be a smooth non-projective hyperk\"ahler manifold of K3 or Kummer type, with $\Pic(X)=\ZZ D$ with $D^2=0$. Let $\gamma$ be the class of a curve on $Y$. Then $\gamma$ is a multiple of a rational curve which moves in a $(2n-2)$-dimensional family.
\end{lemma}
\begin{proof}
We can assume that the divisor $D$ is effective, so that it defines a lagrangian fibration structure on $X$ (see \cite[Theorem 1.3]{Marklagr} and \cite[Corollary 1.1]{matsu}).

 We will for simplicity consider only the case when $X$ has K3-type. The rest of the proof goes through almost verbatim for generalized Kummer varieties.
 
First, we want to deform $Y$ to a general Hilbert scheme $S^{[n]}$ where $S$ is a K3 surface. Let $C=\gamma/r$, $r>0$ be the primitive cohomology class. It suffices to produce a pair $(S^{[n]},H-g\tau)$, where $H-g\tau$ is the class of a rational curve, for every component of the moduli space of pairs $(Y,C)$. In this class of square zero curves (actually, divisors) these components have been determined by Markman \cite[Lemma 2.5]{Marklagr} and Wieneck \cite[Lemma 5.12]{benkum}. All monodromy orbits contain an element of the form $(S^{[n]},H-b\tau)$. In particular, it suffices to take $2(n-1)< b\leq 3(n-1)$. Therefore, we can again use Lemma \ref{allthecurvesweneed} to produce a rational curve of class $H-b\tau$ which moves in a $2n-2$ dimensional family. By Proposition \ref{ranprop}, this curve deforms to the desired curve on $Y$.\end{proof}



\begin{proof}[Proof of Theorem \ref{integralmori}]
Let $X$ be a IHS manifold of K3 or Kummer type and let $C$ be a curve on $X$. If $q(C)>0$, then $C$ is an integral sum of rational curves by Theorem \ref{sumofrat}. If $q(C)=0$, the same holds by Lemma \ref{ratzero} after going to a generic deformation as in Lemma \ref{ratcurve}. We are left with the case $q(C)<0$. If $C$ is extremal, we can apply Lemma \ref{ratnegative} to conclude that $C$ is a multiple of a rational curve with the deformation argument of Lemma \ref{ratcurve}. Otherwise, we can apply the Cone theorem of \cite[Proposition 11]{ht_moving}, so that $C$ can be written as a rational sum of extremal curves $C=\sum a_iR_i$, where all $R_i$ are primitive rational curves by Lemma \ref{ratnegative}. 



We show that the coefficients $a_i$ are in fact integers. Fix an integer $i$ and let us take a divisor $D_i\in \Pic(X)$ which is effective and such that $D_i\cdot R_i>0$, $D_i\cdot R_j< 0$ if $i\neq j$. Such a divisor can be found inside the big cone intersected with the open subsets $\{D\in N^1(X)\,|\,D\cdot R_j<0\}$ and $\{D\in N^1(X)\,|\,D\cdot R_i>0\}$. 

As a small enough multiple of $D_i$ is klt by \cite[Remark 12]{ht_moving}, we can run the MMP to the pair $(X,D_i)$ contracting negative curves in any order, as every MMP terminates by \cite[Theorem 4.1]{lp_mmp}. Thus, there exists a variety $Y$ and a map $X\rightarrow Y$ which contracts all curves whose classes are multiples of the $R_j$ $j\neq i$. Therefore, the class of the pushforward of $C$ to $Y$ is $a_iR_i$, where we keep denoting by $R_i$ its pushforward to $Y$. As $R_i$ is primitive, $a_i\in\ZZ$, and hence the desired conclusion holds.  
\end{proof}
\section{Application to cubic fourfolds}\label{applications}
As an application of Theorem \ref{IHC}, we give a quick proof of the following result, which was proved by Voisin \cite{Voi2}, using a Lefschetz pencil-type argument.
\begin{theorem}\label{cubicIHC}
Let $X\subset \PP^5$ be a cubic fourfold. Then the integral Hodge conjecture holds for $H^4(X,\ZZ)$. In fact, $H^{2,2}(X,\ZZ)$ is generated by classes of rational surfaces.
\end{theorem}

\begin{proof}Let $F=F(X)$ denote the variety of lines on $X$ and let $P\subset F\times X$ denote the incidence correspondence, with projections $p:P\to F$ and $q:P\to X$. We will consider the (transpose of the) Abel--Jacobi map induced by $P$, namely $$\alpha=q_*p^*:H^6(F,\ZZ)\to H^4(X,\ZZ).$$By Beauville--Donagi \cite[Proposition 4]{BD}, this map is an isomorphism. Hence any integral Hodge class $\Gamma\in H^{2,2}(X,\ZZ)$ is the image of a class in $H^{3,3}(F,\ZZ)$, and consequently it is algebraic, by Theorem \ref{IHC}. The last statement also follows, since the incidence correspondence sends rational curves on $F$ to rational surfaces on $X$.\end{proof}
%
%
%
%



\begin{remark}
By Bloch--Srinivas \cite[Theorem 1]{BS} the cycle class map on codimension 2 cycles is injective, so in fact the Chow group $CH^2(X)$ is isomorphic to $H^{2,2}(X,\ZZ)$, and it is generated by cycle classes of rational surfaces. See also the work of Mboro \cite{mboro} for similar statements for $CH_2(X)$ for higher dimensional cubic hypersurfaces.
\end{remark}

%
%
%
%

\begin{remark}

Using results by Shen \cite{Shen} we obtain the following result about the algebraicity of the Beauville--Bogomolov form. Let $X$ and $F$ be as above. The Beauville--Bogomolov form $q$ on $H^2(F,\ZZ)$ defines a class in $H^{12}(F\times F,\ZZ)$ which we will denote by $[q]$. This class is Hodge of type $(6,6)$. Given the validity of the integral Hodge conjecture on $F$, we can rephrase \cite[Proposition 5.5]{Shen} as

\begin{proposition}[Shen]
The class $[q]$ is algebraic if and only if $X$ is $CH_0$-trivial.
\end{proposition}
Here the group $CH_0(X)$ is universally trivial if and only if X admits a Chow-theoretical decomposition of the diagonal (see \cite{Voi2}). See \cite{Shen} for more details.
\end{remark}


\begin{remark}
Mingmin Shen explained to us that one can conversely deduce the integral Hodge conjecture on $F$ from that on $X$ (which holds by Voisin's result \cite{Voisin}). Indeed, in this case  $P_*: CH_1(F)\to CH_2(X)$ is surjective (\cite[Theorem 3.1]{mboro} and \cite[Theorem 1.1]{Shen2} show that $P_*$ is surjective modulo multiples of $h^2$ and a separate argument shows that $h^2$ is also in the image). From this we deduce that the image $\alpha(\gamma)$ of a class $\gamma\in H^{3,3}(F,\ZZ)$ is homologous to a cycle of the form $P_*(\Gamma)$, where $[\Gamma]=\gamma$. In particular $\gamma$ is algebraic.
\end{remark}

{Department of Mathematics, Alma mater studiorum Universit\`a di Bologna, Piazza di Porta san Donato 5, 40126 Bologna, Italy}\\
{{\it Email:} \verb"giovanni.mongardi2@unibo.it"

\bigskip

\noindent {Department of Mathematics, University of Oslo, Box 1053, Blindern, 0316 Oslo, Norway}\\
{{\it Email:} \verb"johnco@math.uio.no"}

\end{document}